\documentclass{compositio}

\usepackage{amsmath,amssymb,amsthm} 
\usepackage[all]{xy}
\usepackage{extarrows}
\usepackage{graphicx}
\usepackage{bm}
\usepackage{comment, url}
\usepackage[normalem]{ulem}

\usepackage[pagewise]{lineno}
\usepackage{wasysym}

\newtheorem{thm}{Theorem}
\newtheorem{lem}[thm]{Lemma}

\newtheorem{prop}[thm]{Proposition}  

\theoremstyle{remark}

\theoremstyle{definition}
\newtheorem{dfn}[thm]{Definition}
\newtheorem{rem}[thm]{Remark} 

\newtheorem{egs}[thm]{Examples}
\newtheorem{eg}[thm]{Example}

\newtheorem{q}[thm]{Question} 

\newtheorem{def/prop}[thm]{Definition/Proposition}

\numberwithin{equation}{section}

\renewcommand\proofname{\bf Proof}

\usepackage[usenames]{color}
\newcommand{\rs}[1]{} \newcommand{\rma}[1]{}

\def\Ker{\mathop{\mathrm{Ker}}\nolimits}

\def\Hom{\mathop{\mathrm{Hom}}\nolimits}

\def\Gal{\mathop{\mathrm{Gal}}\nolimits}

\newcommand{\bb}[1]{{\mathbb{#1}}}
\newcommand{\mca}[1]{{\mathcal{#1}}}

\newcommand{\inj}{\hookrightarrow}
\newcommand{\surj}{\twoheadrightarrow}

\newcommand{\congto}{\overset{\cong}{\to}}

\newcommand{\imp}{\Longrightarrow}

\newcommand{\N}{\bb{N}}
\newcommand{\Z}{\bb{Z}}

\newcommand{\Q}{\bb{Q}}

\newcommand{\R}{\bb{R}}

\newcommand{\ol}{\overline}

\newcommand{\ds}{\displaystyle}

\newcommand{\wt}[1]{{\widetilde{#1}}}

\DeclareMathOperator*{\restprod}%
 {\mathchoice{\ooalign{\ensuremath{\displaystyle\prod}\crcr\ensuremath{\displaystyle\coprod}}}%
             {\ooalign{\ensuremath{\textstyle\prod}\crcr\ensuremath{\textstyle\coprod}}}%
             {\ooalign{\ensuremath{\scriptstyle\prod}\crcr\ensuremath{\scriptstyle\coprod}}}%
             {\ooalign{\ensuremath{\scriptscriptstyle\prod}\crcr\ensuremath{\scriptscriptstyle\coprod}}}%
 }

\title[Chebotarev links are stably generic]
{Chebotarev links are stably generic}

\author{Jun Ueki}
\email{uekijun46@gmail.com}
\address{Department of Mathematics, School of System Design and Technology, Tokyo Denki University\\ 5 Senju Asahi-cho, Adachi-ku, 120-8551, Tokyo, Japan} 

\date{\today}



\begin{document}

\begin{abstract} 
We discuss the relationship between two analogues in a 3-manifold of the set of prime ideals in a number field. 
We prove that if $(K_i)_{i\in \N_{>0}}$ is a sequence of knots obeying the Chebotarev law in the sense of Mazur and McMullen, then $\mca{K}=\cup_i K_i$ is a stably generic link in the sense of Mihara. 
An example we investigate is the planetary link of a fibered hyperbolic finite link in $S^3$. 
We also observe a Chebotarev phenomenon of knot decomposition in a degree 5 non-Galois subcover of an $A_5$(icosahedral)-cover. 
\end{abstract}

\maketitle 


\section{Introduction}

In this article we discuss the relationship between two analogues in a 3-manifold of the set of rational prime numbers. 
We prove that if $(K_i)_{i\in \N_{>0}}$ is a sequence of knots obeying the Chebotarev law in the sense of B.~Mazur and C.~T.~McMullen, then $\mca{K}=\cup_i K_i$ is a stably generic link in the sense of T.\ Mihara. 

We assume that any 3-manifold is the complement of a finite link in an oriented connected closed 3-manifold. 
A knot $K$ in a 3-manifold $M$ means a tame embedding $S^1=\R/\Z\inj M$ or its image with a natural orientation. A link is a countable (finite or infinite) set of disjoint knots. 
For any manifold $X$, we denote the interior of $X$ by ${\rm Int}X$.

McMullen \cite{McMullen2013CM} established a version of the Chebotarev density theorem in which number fields are replaced by 3-manifolds, answering to Mazur's question on the existence of Chebotarev arrangement of knots in \cite{Mazur2012}. 
Their definition is described as follows. 

\begin{dfn}[(Chebotarev law)] \label{defCheb}
Let $(K_i)=(K_i)_{i\in \N_{>0}}$ be a sequence of disjoint knots in a 3-manifold $M$. 
For each $n\in \N_{>0}$ and $j>n$, we put $L_n=\cup_{i\leq n}K_i$ 
and denote the conjugacy class of $K_j$ in $\pi_1(M-L_n)$ by $[K_j]$. 
We say that $(K_i)$ \emph{obeys the Chebotarev law} if 
$$\lim_{\nu \to \infty} \frac{\#\{n<j\leq \nu\mid \rho([K_j])=C\}}{\nu}=\frac{\#C}{\#G}$$
holds for any $n\in \N_{>0}$, any surjective homomorphism $\rho:\pi_1(M-L_n)\to G$ to any finite group, and any conjugacy class $C\subset G$. 
(The left hand side is \emph{the natural density} of $K_i$'s with $\rho([K_j])=C$.)

An infinite link $\mca{K}$ is said to be \emph{Chebotarev} if it obeys the Chebotarev law with respect to some order. 
\end{dfn} 

On the other hand, Mihara \cite{Mihara2019Canada} formulated an analogue of idelic class field theory for 3-manifolds by introducing certain infinite links called stably generic links, refining the notion of very admissible links given by Niibo and the author \cite{Niibo1, NiiboUeki}, and gave a cohomological interpretation of our previous formulation. 
Here we describe the definition of a stably generic link, only using ordinary terminology of low dimensional topology: 

\begin{dfn}[(stably generic link)] \label{defsg} 
Let $M$ be a 3-manifold and $\mca{K}\neq \emptyset$ a link. The link $\mca{K}$ is said to be \emph{generic} if for any finite sublink $L$ of $\mca{K}$, the group $H_1(M-L)$ is generated by components of $\mca{K}-L$. 
The link $\mca{K}$ is said to be \emph{stably generic} if for any finite sublink $L$ of $\mca{K}$ and for any finite branched cover $h:M'\to M$ branched over $L$, the preimage $h^{-1}(\mca{K})$ is again a generic link of $M'$. 
\end{dfn} 

Now we present our main theorem of this article: 
\begin{thm} \label{thmChebSG} 
Let $(K_i)$ be a sequence of disjoint knots in a 3-manifold $M$ obeying the Chebotarev law. Then the link $\mca{K}=\cup_i K_i$ is a stably generic link. 
\end{thm} 

We first carefully observe the behavior of knots in a finite cover which is not necessarily Galois and prove Lemma \ref{1:1} in Section 2. Afterward, we prove Theorem \ref{thmChebSG} in Section 3.

We remark that the set of prime ideals of the ring of integers of a number field is ``stably generic'' 
rather \emph{a priori} (Remark \ref{NT}). Therefore, Theorem \ref{thmChebSG} tells that the Chebotarev law in a 3-manifold implies more than what it does in number theory.\\ 

McMullen proved that sequences of knots $(K_i)=(K_i)_{i \in \N_{>0}}$ given in Examples \ref{eg} below obey the Chebotarev law \cite[Theorems 1.1, 1.2]{McMullen2013CM}. 
The union $\mca{K}=\cup_i K_i$ of such $(K_i)$ is a stable generic link by 
our Theorem \ref{thmChebSG}. 
\begin{egs} \label{eg}
(1) Let $X$ be a closed surface of constant negative curvature, let $M=T_1(X)$ denote the unit tangent bundle, and let $(K_i)$ denote the closed orbits of the geodesic flow in $M$, ordered by length. 

(2) Let $(K_i)$ be the closed orbits of any topologically mixing pseudo-Anosov flow on a closed 3-manifold $M$, ordered by length in a generic metric. 
(We consult \cite{Fenley2008, 
DCalegari2007book} for terminology and basic facts related to pseudo-Anosov flows.)  

\end{egs} 
In \cite{McMullen2013CM}, in order to connect symbolic dynamics to finite branched covers, he 
proved an important lemma that assures that every conjugacy class of  $G$ is presented by a closed orbit, and invoked the notion of a Markov section. 
Then he applied a Chebotarev law for dynamical setting, which was proved by Parry--Pollicott \cite[Theorem 8.5]{ParryPollicott1990} with use of a method of Artin $L$-functions. 
We note that special cases of (1) and (2) were initially proved to obey the Chebotarev law by Adachi and Sunada in \cite[Proposition II-2-12]{Sunada1984ASPM} and \cite[Proposition C]{AdachiSunada1987JFA}, the latter being related to topological entropy. 

In Section 4, we examine an example contained in (2) called \emph{the planetary link} $\mca{K}$ of a fibered hyperbolic finite link $L$ in $S^3$, that is, the periodic orbits of the suspension flow of the monodromy map of $L$. 
By virtue of McMullen's theorem \cite[Theorem 1.2]{McMullen2013CM} together with the Nilsen--Thurston uniformization theorem (\cite[Theorem 0.1]{WPThurston1986HS2}), the union $\mca{K}\cup L$ obeys the Chebotarev law, if ordered by length (Proposition \ref{eg2}). 
For any finite sublink $L'\subset \mca{K}$, the union $L\cup L'$ 
is again hyperbolic (Proposition \ref{hyperbolic}). 
Moreover, such $\mca{K}\cup L$ is \emph{stably Chebotarev} (Proposition \ref{sC}).

In addition to the Chebotarev law, the planetary link sometimes has another very noteworthy property. 
Ghrist and others proved that if a link $L$ belongs to a certain large class of links containing the figure-eight knot, then the planetary link of $L$ contains every links 
\cite{Ghrist1997, GhristHolmesSullivan1997book, GhristKin2004}. 
Hence we have a sequence $(K_i)$ of knots containing every isotopy class of links and obeying the Chebotarev law (Proposition \ref{eg3}). 
Moreover, as formulated by Kopei in \cite{Kopei2006}, this example satisfies an analogue of the product formula $|a|_\infty \prod_p |a|_p=1$ $(a\in \ol{\Q})$, where $p$ runs through all the prime numbers and $|a|_p$ denote the $p$-adic norm with $|p|_p=p^{-1}$.  
Therefore, the planetary link would give a fundamental setting, when we establish an analogue of number theory on 3-manifolds. 

In Section 5, we further study non-Galois covers, and observe an example of Chebotarev phenomena in an analogue of a quintic field. 
We first prove the coincidence of the decomposition type of a knot and the cycle type of the monodromy permutation (Proposition \ref{Artin}), which is an analogue of Artin's argument in \cite{Artin1923MA}. 
Then we examine the density of knots in a Chebotarev link of each decomposition type in a degree 5 subcover of an $A_5$(icosahedral)-cover. \\

For a sequence of knots ordered by length and obeying the Chebotarev law, we may define analogues of Artin $L$-functions 
(cf.~\cite{
Sunada1984ASPM, 
AdachiSunada1987JFA, Sunada1988, ParryPollicott1990}). 
In addition, Mihara's refinement allows us to study analogues of ray class fields. 
We expect that Theorem \ref{thmChebSG} 
would play a key role to expand an analogue of id\`{e}lic class field theory for 3-manifolds, in a direction of analytic number theory, with ample interesting examples.  
Another analogue in a more general setting are due to 
J.~Kim (see ver.1 of \cite{JKim2018}) and others \cite{KimMorishitaNodaTerashima2021}. 

\section{Knots in a finite non-Galois cover}

In this section, we carefully observe the behavior of knots in a finite cover which is not necessarily Galois (regular), and obtain the following lemma. 

\begin{lem} \label{1:1} 
Let $h:N\to M$ be a finite (unbranched) cover of 3-manifolds and $K\subset M$ a knot. 

{\rm (1)} Let $K'$ be a connected component of $h^{-1}(K)$ 
in $N$. If the restriction map 
$h|_{K'}:K'\to K$ is a bijection, then the conjugacy classes satisfy $h_*([K'])\subset [K]$ in $\pi_1(M)$. 

{\rm (2)} If $k \in [K]\cap h_*(\pi_1(N)) \neq \emptyset$, then there exists some connected component $K'$ of $h^{-1}(K)$ such that $h|_{K'}:K'\to K$ is a bijection and $k \in h_*([K'])$ holds. 
\end{lem}

We have an analogue of the Hilbert ramification theory for Galois branched covers of 3-manifolds, in which we describe the behavior of prime ideals and knots using the language of fundamental groups (\cite{Ueki1}). 
For a non-Galois cover $h:N\to M$ and a knot $K$ in $M$, 
the covering degrees of the restriction maps $h|_{K'_i}:K'_i\to K$ for components of $h^{-1}(K)=\cup_i K'_i$ does not necessarily coincide with each other. 
In order to prove the theorem, we need to care such a case. 

We first prepare some basics. 
We assume that any 3-manifold $X$ is equipped with a base point $b_X$. 
A \emph{path} $\gamma=\gamma(t)$ in $X$ is a continuous map $\gamma:[0,1]\to X$ or its image. 
We denote \emph{the inverse path} $\ol{\gamma}$ of $\gamma$, which is defined by $\ol{\gamma}(t):=\gamma(1-t)$. 
A \emph{loop} $l=l(t)$ in $X$ is a path satisfying $l(0)=l(1)=b_X$. 

Let $K$ be a knot in a 3-manifold $X$ and let $\gamma$ be a path with $\gamma(0)=b_X$ and $\gamma(1)\in K$. Then a loop $l_\gamma$ is obtained as the composite of $\gamma$, $K$, and $\ol{\gamma}$. Each element $k$ of the conjugacy class $[K]$ of $K$ in $\pi_1(X)$ is presented by such a loop $l_\gamma$ for some $\gamma$. 

If $h:N \to M$ is a cover of 3-manifolds, we assume $h(b_N)=b_M$ 
so that there is a natural injective homomorphism $h_*: \pi_1(N)\to \pi_1(M)$. 
Let $g:(Z,b_Z)\to (M,b_M)$ be a continuous map from a connected compact manifold with a base point.  
A continuous map $\wt{g}:Z\to N$ satisfying $h\circ \wt{g}=g$ is called a \emph{lift} of $g$ with respect to $h$. 
We have a well-known lifting criterion (See \cite[Propositions 1.33, 1.34]{HatcherAT}): 

\begin{lem}
\label{lifting} Let $h:(N,b_N)\to (M,b_m)$ be a (unbranched) cover of 3-manifolds (here we explicitly write the base points) and let $g:(Z,b_Z)\to (M,b_M)$ be a continuous map from a connected compact manifold with a base point. 
Then there exists a lift $\wt{g}$ of $g$ if and only if $g_*(\pi_1(Z))< h_*(\pi_1(Y))$ holds. For each $b\in h^{-1}(b_M)$, 
a lift $\wt{g}$ satisfying $\wt{g}(b_Z)=b$ is unique. 
\end{lem}

\begin{proof}[\textbf{Proof of Lemma \ref{1:1}}] (1) Let $k' \in [K']$ and let $\gamma'$ be a path in $M$ with $\gamma'(0)=b_N$ and $\gamma'(1) \in K'$ such that the composite loop $l_{\gamma'}$ of $\gamma'$, $K$, and $\ol{\gamma'}$ presents $k'$. 
Put $\gamma:=h\circ \gamma'$. Then we have $\gamma(0)=b_M$ and $\gamma(1)\in K$ holds, and the composite loop $l_\gamma$ of $\gamma$, $K$, $\ol{\gamma}$ presents some $k\in [K]$. 
If $f$ is the covering degree of $h|_{K'}:K'\to K$, then $h_*(k')=k^f$ holds. 
By the assumption we have $f=1$, and hence $h_*(k')=k$ holds. Since any conjugate of $k'$ maps to a conjugate of $k$, we have $h_*([K'])\subset [K]$. 

(2) Let $k\in [K]\cap h_*(\pi_1(N))$. Then there is a path $\gamma$ in $M$ with $\gamma(0)=b_M$ and $\gamma(1)\in K$ such that the composite loop $l_\gamma$ of $\gamma$, $K$, and $\ol{\gamma}$ presents $k$. 
Let $\wt{l}_\gamma$ denote the lift of $l_\gamma$ with $\wt{l}_\gamma(0)=b_N$. 
Since $k\in h_*(\pi_1(N))$, the lifting criterion (Lemma \ref{lifting}) assures that $\wt{l}_\gamma$ is again a loop. 
Let $\wt{\gamma}$ denote the lift of $\gamma$ with $\wt{\gamma}(0)=b_N$ and 
let $K'$ denote the connected component of $h^{-1}(K)$ such that $\wt{\gamma}(1)\in K'$. 
Then $\wt{l}_\gamma$ is the composite loop of $\wt{\gamma}$, $K'$, and $\ol{\wt{\gamma}}$. Therefore the restriction map $h|_{K'}:K'\to K$ is a bijection 
and $k\in h_*([K'])$ holds.
\end{proof}

\section{Proof of the theorem} 
In order to prove Theorem \ref{thmChebSG}, we first introduce the notion of a weakly Chebotarev link: 

\begin{dfn} Let $M$ be a 3-manifold and $\mca{K}=\cup_{i\in \N_{>0}} K_i$ a countable link. The link $\mca{K}$ in $M$ is said to be  \emph{weakly Chebotarev} if for any surjective homomorphism $\rho:\pi_1(M)\surj G$ to any finite group, any conjugacy class $C$ of $G$ is the image $\rho([K_i])$ of the conjugacy class $[K_i] \subset \pi_1(M)$ of some component $K_i$ of $\mca{K}$. 
\end{dfn}

\begin{lem} \label{WCG}
If a link $\mca{K}=\cup_i K_i$ in a 3-manifold $M$ is weakly Chebotarev, then $H_1(M)$ is generated by components of $\mca{K}$.
\end{lem} 

\begin{proof}  
Otherwise there is a surjective homomorphism $\pi_1(M)\surj H_1(M)\surj \Z/p\Z$ for some prime number $p$ such that the image of any conjugacy class $[K_i]$ is zero, being contradiction. 
\end{proof}

The following is a key lemma to prove the theorem. 

\begin{lem} \label{WCihnerit}
If $\mca{K}=\cup_i K_i$ in $M$ is weakly Chebotarev, 
then for any finite (unbranched) cover $h:N\to M$, the preimage $\mca{K}'=\cup_j K'_j$ of $\mca{K}$ is again weakly Chebotarev. 
\end{lem} 

\begin{proof} 
Let $\rho:\pi_1(N)\surj G$ be any surjective homomorphism onto a finite group and let $C$ be any conjugacy class of $G$. We will prove that there exists some $K'_j \subset N$ 
satisfying $C=\rho([K'_j])$. 
Let $L'$ be a loop in $N$ with $C=\rho([L'])$. 
We may assume that the restriction map $h|_{L'}$ is injective, by moving $L'$ a little if necessary. 
If we put $L:=h(L')$, then we have $h_*([L'])\subset [L]$ by Lemma~\ref{1:1}~(1). 
Let $\Gamma$ denote the maximal normal subgroup of $\pi_1(M)$ contained in $\Ker(\rho)$.  Then $\Gamma$ is again of finite index. 
Let $q:\pi_1(M) \surj \pi_1(M)/\Gamma$ and $q':\pi_1(N) \surj \pi_1(N)/\Gamma$ denote the quotient maps. 
Since $\mca{K}=\cup K_i$ is weakly Chebotarev and $q$ is a surjective homomorphism onto a finite group, we have $q([L])=q([K_i])$ for some $K_i$. 
Let $l \in [L']$. Since $h_*([L'])\subset [L]$, we have $h_*(l)=kg$ for some $k\in [K_i]$ and $g\in \Gamma$. 
Since $k=h_*(l)g^{-1}\in h_*(\pi_1(N))$, by Lemma \ref{1:1} (2), 
there exists some connected component $K'_j$ of $h^{-1}(K_i)$ such that 
the restriction $h|_{K'_j}:K'_j\to K$ is a bijection and $k \in h_*([K'_j])$ holds. 
Hence $q'(l')$ is a common element of $q'([L'])$ and $q'([K'_j])$ in $\pi_1(N)/\Gamma$. 
Since $[L']$ and $[K'_j]$ are conjugacy classes in $\pi_1(N)$, so are $[L']$ and $[K'_j]$ in $\pi_1(N)/\Gamma$. 
Therefore we have 
$q'([L'])=q'([K'_j])$ in $\pi_1(N)/\Gamma$. Since $\rho$ factors through $q'$, we obtain $C=\rho([L'])=\rho([K'_j])$~in~$G$. 
\end{proof}

\begin{proof}[\textbf{Proof of Theorem \ref{thmChebSG}}] 
Suppose that a sequence $(K_i)_{i\in \N_{>0}}$ of knots in a 3-manifold $M$ satisfies the Chebotarev law. 
Let $n\in \N_{>0}$ and put $L_n=\cup_{i\leq n} K_i$. 
Then the link $\mca{K}-L_n$ in $M-L_n$ is weakly Chebotarev. 

Let $h:N\to M$ be any finite branched cover branched along some finite sublink $L$ of $\mca{K}$ in $M$, and let $L'$ be any finite sublink of the preimage $h^{-1}(\mca{K})$ in $N$. Then the union $L\cup h(L')$ is contained in $L_n$ for some $n$. 
Put $X=M-L_n$ and $Y=N-h^{-1}(L_n)$, and let $h:Y\to X$ denote the restriction of $h$ to the exteriors. 
Since $\mca{K}-L_n$ in $M-L_n$ is weakly Chebotarev, by Lemma \ref{WCihnerit}, the preimage $h^{-1}(\mca{K}-L_n)$ in $Y$ is again weakly Chebotarev. 
By Lemma \ref{WCG}, components of $h^{-1}(\mca{K}-L_n)$  generates $H_1(Y)$, and hence components of $h^{-1}(\mca{K})-L'$ generates $H_1(N)$. 
Therefore $h^{-1}(\mca{K})$ is a generic link of $N$, and 
$\mca{K}$ is a stably generic link. 
\end{proof} 


We remark that Theorem 3 was initially claimed by McMullen in his e-mail, with an idea of the proof above. 
Mihara also gave several essential comments to refine the proof.

\begin{rem} \label{NT} Let us examine the condition of ``stably generic'' in number theory. 
Let $k$ be any number field and $S$ a finite set of prime ideals of the ring $\mca{O}_k$ of integers of $k$. 
Then by virtue of \cite[Chapter VI, Theorems 6.6 and 7.1]{Neukirch} connecting ideal theoretic and id\`ele theoretic class field theories via the notion of ray class groups, 
the Galois group $\Gal(k_S^{\rm ab})$ of the maximal abelian extension of $k$ unramified outside $S$ is topologically generated by the set of images of (the Frobenius elements of) prime ideals outside $S$, 
namely, the set ${\rm Spec}\mca{O}_k$ of prime ideals  of $\mca{O}_k$ is ``generic''. 
Since this condition holds for any finite extension $F/k$, we see without any additional argument that the set ${\rm Spec}\mca{O}_k$ is ``stably generic''. 
This fact is based only on an algebraic side of number theory, apart from an analytic side containing the Chebotarev law. In this sense, Theorem \ref{thmChebSG} tells that the Chebotarev law in a 3-manifold implies more than what it does in number theory.  
\end{rem}

\section{The planetary link}

In this section, we investigate the planetary link of a fibered hyperbolic finite link $L$ in $S^3$, which is an example of Chebotarev link, and ask several questions. 
Typical examples of fibered hyperbolic links are the figure eight knot, the Whitehead link, and the Borromean ring.

We first define \emph{the planetary link} $\mca{K}$ of a fibered finite link $L$ in $S^3$, following Birman and Williams \cite{BirmanWilliams1983CM}. 

\begin{dfn} A finite link $L$ in $S^3$ is fibered if there is a fibration $f:S^3-L\to S^1$ with fiber the interior of a Seifert surface $\Sigma$. This induces a self diffeomorphism $\varphi$ on $\Sigma$ called the monodromy. The suspension flow of $\varphi$ on $S^3-L$ is equivalent to that obtained by integrating the gradient $\nabla f$. 
By Thurston's classification theorem of surface diffeomorphisms \cite{WPThurston1986HS2}, we may speak of ``the'' fibration of $L$. 
\emph{The planetary link} $\mca{K}$ of a fibered finite link $L$ is the set of periodic orbits of this flow. 
\end{dfn} 

The following proposition is a partial generalization of \cite[Theorem 2.3]{NiiboUeki}. It is obtained similarly to the note just below \cite[Corollary 1.3]{McMullen2013CM}, as a corollary of \cite[Theorem~1.2]{McMullen2013CM}. 

\begin{prop}\label{eg2} Let $L$ be any fibered hyperbolic finite link in $S^3$ and let $\mca{K}$ denote the planetary link of $L$. 
Then the union $\mca{K}\cup L$ obeys the Chebotarev law, if ordered by length with respect to a generic metric. 
\end{prop}

\begin{proof}
A celebrated theorem by W.~P.~Thurston \cite[Theorem 0.1]{WPThurston1986HS2} asserts that for an automorphism $\varphi$ of a compact surface $\Sigma$ of negative Euler characteristic, the interior of the mapping torus $M_\varphi=\Sigma\times \R/(x,t)\sim (\varphi(x),t+1)$ admits a complete hyperbolic structure of finite volume if and only if $\varphi$ is a pseudo-Anosov map.

Since $L$ is a fibered link, the complement $S^3-L$ is homeomorphic to the interior of the mapping torus $M_\varphi$ of the monodromy map $\varphi$ on a Seifert surface $\Sigma$ of $L$. 
Since $L$ is hyperbolic, ${\rm Int}M_\varphi\cong S^3-L$ admits a (unique) complete hyperbolic structure, and  Thurston's theorem assures that $\varphi$ is a pseudo-Anosov map. 
Therefore the suspension of $\varphi$ is a pseudo-Anosov flow on $S^3-L$. 

Since the boundary of $\Sigma$ is stable under $\varphi$ (but not fixed), the flow on $S^3-L$ naturally extends to $S^3$, and $L$ itself becomes a set of closed orbits. 
Since the Chebotarev law persists under Dehn surgeries\footnote{To be more precise, some rational filling yields a pseudo-Anosov so that \cite[Theorem 1.2]{McMullen2013CM} applies and then the $\infty$-surgery results an expected link in $S^3$.}, now \cite[Theorem 1.2]{McMullen2013CM} asserts that the set $(K_i)$ of closed orbits of the flow on $S^3$ obeys the Chebotarev law, if ordered by length 
with respect to a generic metric. 
\end{proof}

The following Proposition is obtained in a similar way to Miller \cite[Proposition 4.2]{Miller2001EM}. 
\begin{prop} \label{hyperbolic} 
Let $L$ be a fibered hyperbolic finite link in $S^3$ and let $\mca{K}$ denote the planetary link of $L$. 
Then for any finite sublink $L'$ of $\mca{K}$, the union $L\cup L'$ is again hyperbolic. 
\end{prop} 
\begin{proof}  
Let $\Sigma$ be a Seifert surface of $L$ and let $\varphi$ denote the monodromy map on $\Sigma$. 
Put $\Sigma':=\Sigma-L'$ and $\Sigma'':=\Sigma-{\rm Int}V_{L'}$, where $V_{L'}$ is a tubular neighborhood of $L'$. 
Since $L'$ consists of periodic orbits of $\varphi$, the map $\varphi$ restricts to $\Sigma'$. 
Let $\psi:\Sigma'\congto {\rm Int}\Sigma''$ be a homeomorphism, and put $\varphi':=\psi\circ \varphi\circ\psi^{-1}$. 
Then $h'$ is a pseudo-Anosov map on the interior of $\Sigma''$, which naturally extends to the boundary of $\Sigma''$. 
Let $M'_\varphi$ and $M''_{\varphi'}$ denote the mapping tori of $\varphi$ and $\varphi'$ acting on $\Sigma'$ and $\Sigma''$ respectively.  
Then we have homeomorphisms among $S^3-L\cup L'$, ${\rm Int}M'_\varphi$, and ${\rm Int}M''_{\varphi'}$. 
Since $\varphi'$ is a pseudo-Anosov map on $\Sigma''$, again by virtue of Thurston's theorem \cite[Theorem 0.1]{WPThurston1986HS2}, the interior ${\rm Int}M''_{\varphi'}$ admits a complete hyperbolic structure, and hence so does $S^3-L\cup L'$. 
\end{proof} 

In what follows we introduce the notion of a stably Chebotarev link, which is slightly stronger to that of a Chebotarev link, 
and give another partial proof of Theorem \ref{thmChebSG}. 

\begin{dfn} We say a sequence $(K_i)_i$ of knots in a 3-manifold $M$ is \emph{stably Chebotarev} if for any finite branched cover $h:N\to M$ branched along some finite sublink $L$ of $\mca{K}=\cup_i K_i$, the preimage $h^{-1}(\mca{K}-L)=\cup_j K'_j$ with some order obeys the Chebotarev law.
\end{dfn} 

\begin{prop} \label{sC} Any link obeying the Chebotarev law in Examples \ref{eg} (2) due to \cite[Theorem 1.2]{McMullen2013CM} is a stably Chebotarev link. 
\end{prop} 

\begin{proof}
Let $(K_i)$ be a set of knots obeying the Chebotarev law in Examples \ref{eg}. 
Let $h:N\to M$ a finite branched cover branched along a finite sublink $L$ of $\mca{K}:=\cup_i K_i$ and let $h:Y\to X$ denote the restriction to the complement of $h^{-1}(L)$. 
Since $L$ consists of periodic orbits, the topologically mixing pseudo-Anosov flow on $X$ lifts to $Y$ and extends to $N$. 
Since the Chebotarev law persists under Dehn surgeries, now \cite[Theorem 1.2]{McMullen2013CM} asserts that $\mca{K}'$ again obeys the Chebotarev law if it is ordered by length. 
\end{proof} 

Since a Chebotarev link is generic, a stably Chebotarev link is stably generic. Therefore, Proposition \ref{sC} reproves Theorem \ref{thmChebSG} for links given in Example \ref{eg} (2). 

\begin{q} We wonder whether the following are true. 

(1) Any sequence of knots $(K_i)$ in a 3-manifold $M$ obeying the Chebotarev law is stably Chebotarev. 

(2) In addition, let $h:N\to M$ be a finite branched cover. Let $\mca{K}'=\cup K_j$ be the preimage of $\mca{K}=\cup K_i$ with suffix $j\in \N_{>0}$ and denote $h(j)=i$ if $h(K'_j)=K_i$. Then $(K'_j)$ again satisfies the Chebotarev law if 
(i) $(K'_j)$ is ordered so that $k<l \imp h(k)\leq h(l)$ holds, or if 
(ii) $(K'_j)$ is again ordered by length. 
\end{q} 

Birman and Williams \cite{BirmanWilliams1983CM} conjectured that the planetary link of the figure-eight knot does not contain a figure-eight knot. However, Ghrist and others proved that for a link $L$ in a large class of links containing the figure-eight link, the planetary link of $L$ contains every isotopy class of links,  by developing a theory of universal template (\cite[Theorem 4]{Ghrist1997}, \cite[Remark 3.2.20]{GhristHolmesSullivan1997book}, \cite{GhristKin2004}). 
Their class contains the figure-eight link, the Whitehead link, the Borromean ring, and every fibered non-torus 2-bridge  knot. 
Therefore we obtain another generalization of \cite[Theorem 2.3]{NiiboUeki}: 
\begin{prop} \label{eg3} 
There exits a sequence $(K_i)$ of knots containing every isotopy class of links and obeying the Chebotarev law. 
\end{prop} 

Ghrist and Kin conjectured in \cite{GhristKin2004} that if the monodromy of $L$ is ``too twisted'', then the planetary link of $L$ does not contain every link. 
There are several other ways to construct a universal template, giving an infinite link containing every link (e.g., \cite{Kin2000JKTR}). 
Now we would like to (re-)ask the following questions, due to our interests in Arithmetic Topology. 

\begin{q} 
(1) Does the planetary link of any fibered hyperbolic finite link $L$ contain every isotopy class of links? (This is asked in \cite{GhristKin2004}.)

(2) Is there any other example of an infinite link obtained from a universal template and obeying the Chebotarev law? 
\end{q} 

\section{Decomposition type of knots} 

In this section, we extend the study on non-Galois covers in Section 2 to describe the decomposition types of knots in a  group theoretical way,
and apply the Chebotarev law for an analogue of a quintic field, that is, a non-Galois number field extension of degree 5 with the Galois group being $A_5$. 

\begin{dfn}
(1) For each $n\in \N_{>0}$, 
we denote the $n$-th symmetric group by $S_n$ and the $n$-th alternating group by $A_n$. \emph{The cycle type of} $\sigma \in S_n$ is $(f_1,\cdots, f_r)$ if it is the product of disjoint cycles of length $f_1\geq \cdots\geq f_r$.

(2) For a finite cover $h:N\to M$ and a knot $K\subset M$, \emph{the decomposition type} of $K$ in $h$ is $(f_1,\cdots, f_r)$ if the inverse image of $K$ consists of $r$ components as $h^{-1}(K)=\cup_{1\leq i\leq r}K_i$ with $f_i={\rm deg}(h:K_i\to K)$ for $i=1,\cdots, r$ and $f_1\geq \cdots \geq f_r$. 
\end{dfn}

Let $h:(N,b_N)\to (M,b_M)$ be a finite (unbranched) cover, which is not necessarily Galois. 
In order to make our argument clear, we recall an important left action of $\pi_1(M)$ on $h^{-1}(b_M)=\{b_1,\cdots,b_n\}$ with $b_1=b_N$ called the monodromy action, which is different from the natural Galois action on the left defined only if $h$ is Galois.  (A standard reference is \cite[p.68--70]{HatcherAT}.)

For a homotopy class $\gamma$ of paths with fixed endpoints, we denote by $\gamma(0)$ and $\gamma(1)$ the starting point and the terminal point respectively. 
The universal cover $\wt{M}\to M$ is identified with the space of homotopy classes $\gamma$ of paths 
with fixed endpoints and with $\gamma(0)=b_M$. 
Put $H:=h_*(\pi_1(N))<\pi_1(M)$ and identify $N$ with the space $H\backslash \wt{M}$ of right cosets $H\gamma$ with $\gamma \in \wt{M}$. 
For each $b_j$, take $\gamma_j\in \wt{M}$ with $b_j=H\gamma_j$, that is, the lift $\wt{\gamma}_j$ to $N$ with $b_1=\wt{\gamma}_j(0)$ satisfies $b_j=\wt{\gamma}_j(1)$. 
Then \emph{the monodromy action} is defined as the composite of the isomorphism $\pi_1(M)\congto \pi_1(M)^{\rm op}=\pi_1(M);\gamma \mapsto \gamma^{-1}$ and the natural transitive right action $(H\gamma_j)\gamma=H\gamma_j\gamma$ for each $\gamma \in \pi_1(M)$. 
This action induces a homomorphism $\rho:\pi_1(M)\to {\rm Aut}(h^{-1}(b_M))\cong S_n$ called \emph{the monodromy permutation.} 
We have $\rho(\gamma)(b_j)=H\gamma_j\gamma^{-1}$.

Let $K\subset M$ be a knot with $b_M\in K$, put $h^{-1}(K)=\cup_{1\leq i \leq r}K_i$, and let $z\in \pi_1(M)$ denote the homotopy class presented by the oriented loop $K$. 
For any $i$ and $j$ with $b_j\in K_i$, the point $\rho(z)(b_j)$ is the terminal point of the lift of $\gamma_j z^{-1}$ starting at $b_1$, hence $\rho(z)(b_j) \in K_i$ holds. 
Note that the lift  $\kappa=\kappa(t)$ to $N$ of the inverse of the loop $K$ with $\kappa(0)=b_j$ satisfies $\kappa(1)=\rho(z)(b_j)$. 

If $H$ is a normal subgroup of $\pi_1(M)$, on the other hand, \emph{the Galois action} of $\pi_1(M)$ on $N$ is defined to be the natural left action given by 
$\gamma' (H\gamma)=\gamma'H\gamma'^{-1}\gamma'\gamma=H\gamma' \gamma$, 
inducing the identification $\Gal(h)\cong h_*(\pi_1(N))\backslash \pi_1(M) = \pi_1(M)/h_*(\pi_1(N))$. 
For $K,z,b_j$, and $K_i$ being as above, the point $z b_j$ is well-defined and is the terminal point of the lift of $z\gamma_j$ starting at $b_1$, so that $z b_j$ is not necessarily on $K'$. 

We may replace $\wt{M}\to M$ and $\pi_1(M)$ by a finite Galois cover $\wt{h}:W\to M$ factoring through $h:N\to M$ and $G=\Gal(h)$ in the definitions and argument above. 
Now we easily obtain the following proposition, which is an analogue of Artin's argument in \cite{Artin1923MA} (see also \cite[Chapter 16.2, Theorem 2]{Takagi1948book}). 

\begin{prop} \label{Artin} 
Let $\wt{h}:W\to M$ be a finite (unbranched) Galois cover with $G=\Gal(\wt{h})$, and $h:N\to M$ a subcover of degree $n$, which is not necessarily Galois. 
Let $\rho:G\to S_n$ denote the monodromy permutation induced by putting $h^{-1}(b_M)=\{b_1,b_2,\cdots,b_n\}$ with $b_1=b_N$. 

Let $K\subset M$ be a knot and let $z\in G$ be an element of the image of the conjugacy class $[K]\subset \pi_1(M)$ of $K$ under the natural homomorphism $\pi_1(M)\surj \pi_1(M)/\wt{h}_*(\pi_1(W))\cong G$.
Then the cycle type of $\rho(z)$ is $(f_1,\cdots, f_r)$ if and only if the decomposition type of $K$ in $h$ is $(f_1,\cdots, f_r)$. 
\end{prop}

\begin{proof} 
Since the cycle type of $\sigma\in S_n$ is stable under conjugate, 
we may assume $b_M\in K$ and that $z$ is presented by the loop $K$. Let $h^{-1}(K)=\cup_{1\leq i\leq r}K_i$. 
The observation of the monodromy permutation above assures that $\rho(z)$ permutes $b_j$'s on each $K_i$ cyclically, and the length of each cycle in the cycle decomposition of $\rho(z)$ coincides with $f_i={\rm deg}(h:K_i\to K)$. 
\end{proof} 

\begin{rem} 
Let $H<G$ denote the subgroup corresponding to $h:N\to M$. 
Suppose that $w/k$ is a finite Galois extension of number fields with $G=\Gal(w/k)$ and $l/k$ is the fixed subextension of $H<G$. Then an analogue of $h^{-1}(b_M)=\{b_1,\cdots,b_n\}$ is the set of conjugates of $l/k$, or more precisely, the set $\Hom_k(l, w)$. 
The monodromy permutation $G\to S_n$ of $l/k$ is defined in a similar way, at least in a group theoretic sense. 
The element $z$ in the proof above is an analogue of the Frobenius element of an unramified prime, generating the decomposition group of the component $K'$ of $\wt{h}^{-1}(K)$ with $b_W \in K'$. 
Noting that the subcover $h':W\to N$ is Galois and using the notion of \emph{the decomposition groups} in the Hilbert ramification theory (\cite[Chapter 5]{Morishita2012}, \cite{Ueki1}), we may describe the proof in a more parallel way to the one in number theory. 
\end{rem}

\begin{eg} \label{egA5}
Let $\mca{K}=\cup_i K_i$ be the planetary link of the figure-eight knot in $S^3$, obeying the Chebotarev law. Let $L$ be a trefoil 
in $\mca{K}$, and put $M=S^3-L$. 
We have a well-known surjective homomorphism $\tau:\pi_1(M)\surj A_5$. 
Let $\wt{h}:W\to M$ denote the corresponding $A_5$-cover. 
(Then the Fox completion $\ol{W}$ of $W$ is a Poincar\'e 3-sphere. cf.~\cite{Rolfsen1990}.) 
Let $H<\Gal(h)\cong A_5$ be any subgroup of index 5, and let $h:N\to M$ denote the corresponding subcover of degree 5. 
The kernel ${\rm Ker}(\rho)$ of the monodromy permutation $\rho:A_5\to S_5$ coincides with the normalizer of $H$ in $A_5$. 
Since $A_5$ is a simple group, ${\rm Ker}(\rho)=\{{\rm id}\}$ holds and $\rho$ is an injection. 
By Proposition \ref{Artin} together with the Chebotarev law applied for the composite 
$\rho':\pi_1(M)\underset{\tau}{\surj} A_5\underset{\rho}{\congto} {\rm Im}(\rho)$, 
the number of elements of $A_5$ of each cycle type and the natural density of $K_i$'s of each decomposition type are given as follows. 

\begin{center}
\begin{tabular}{|c||c|c|c|c|}
\hline (cycle/decomposition) type&(1,1,1,1,1)&(2,2,1)&(3,1,1)&(5)\\ 
\hline \hline number of elements of $A_5$ &1&15&20&24\\ 
\hline 
density of knots in $\mca{K}$ &1/60&1/4&1/3&2/5\\
\hline 
\end{tabular} 
\end{center}
Here, a knot of decomposition type $(1,\!1,\!1,\!1,\!1)$ is totally decomposed and that of $(5)$ is totally inert. 
\emph{The natural density} of $K_i$'s with property $P$ is defined by 
$\ds \lim_{\nu\to \infty}\frac{\#\{i\leq \nu\mid K_i{\rm \ satisfies\ }P\ \}}{\nu}$. 
\end{eg} 

Artin $L$-functions of symbolic flows due to Parry--Pollicott \cite{ParryPollicott1990} are regarded as that of Chebotarev links. 
We may also discuss an analogue of Artin's argument in \cite{Artin1923MA} with use of $L$-functions associated to the setting in Example \ref{egA5}. 

\section*{Acknowledgments} 
I would like to express my sincere gratitude to  Noboru Ito, Curtis T.~McMullen, Tomoki Mihara, Masanori Morishita, Hirofumi Niibo, 
and people who attended my talks in several places 
including CIRM in Luminy, Tambara seminar house in Gunma, Noda, Akita, Kanazawa, Peking, and Osaka for inspiring conversation. 
I am grateful to the anonymous referees and experts of journals for careful reading and sincere comments. 
I also would like to thank Jun'ich Akama at TIE 
for serving a place for research in the summer recess. 


\bibliographystyle{amsalpha}
\bibliography{ju.Cheb3.7.bbl}%



\newpage 

\section*{Corrigendum: `\emph{Chebotarev links are stably generic} \cite{ueki7}'} 
\begin{itemize}
\item The definition of $L_n$ in \cite[Definition 1]{ueki7} ought to be $L_n=\cup_{i\leq n}K_i$, instead of $L_n=\cup_{i\leq n}K_n$. 
\item In the proofs of \cite[Proposition 12 and 15]{ueki7}, just before applying \cite[Theorem 1.2]{McMullen2013CM}, the following sentence ought to be inserted; `Since the Chebotarev law persists under Dehn surgeries'. 
To be more precise, some rational filling yields a pseudo-Anosov so that \cite[Theorem 1.2]{McMullen2013CM} applies and then the $\infty$-surgery results an expected link in $S^3$. An extended argument will appear elsewhere. 
\end{itemize} 

These descriptions are corrected in this version on the arXiv. 

\end{document}